\documentclass[a4paper,11pt,reqno]{amsart}
\usepackage[T1]{fontenc}
\usepackage[utf8]{inputenc}
\usepackage{lmodern}
\usepackage{amsmath, amsthm, amssymb,amscd, mathrsfs, amsfonts, mathtools}
\usepackage{hyperref}
\usepackage{euler}
\usepackage{times}
\usepackage[all]{xy}
\usepackage{todonotes}
\usepackage{xcolor}

\usepackage[lite]{amsrefs}

\renewcommand{\PrintDOI}[1]{\href{http://dx.doi.org/\detokenize{#1}}{doi: \detokenize{#1}}%
	\IfEmptyBibField{pages}{, (to appear in print)}{}}

\newcommand{\Z}{\mathbb{Z}}

\newcommand{\rt}{\vartriangleright}

\newtheorem{theorem}{Theorem}[section]
\newtheorem{corollary}[theorem]{Corollary}

\newtheorem{proposition}[theorem]{Proposition}

\newtheorem{definition}[theorem]{Definition}

\newtheorem{example}[theorem]{Example}

\newtheorem{remark}[theorem]{Remark}

\title{$f$-Racks, $f$-Quandles, their Extensions and Cohomology }

\author{Indu R. U. Churchill}
\address{Department of Mathematics,
	University of South Florida, Tampa, FL 33620, U.S.A.}
 \email{udyanganiesi@mail.usf.edu}

\author{M. Elhamdadi}
\address{Department of Mathematics,
	University of South Florida, Tampa, FL 33620, U.S.A.}
\email{emohamed@math.usf.edu}

 \author{M. Green}
 \address{Department of Mathematics,
 	University of South Florida, Tampa, FL 33620, U.S.A.}
 \email{mjgreen@mail.usf.edu}

 \author{A. Makhlouf}
 \address{Laboratoire de Math\'ematiques, Informatique et Application. 
 	Universit\'e de Haute Alsace, France}
 \email{Abdenacer.Makhlouf@uha.fr}

\begin{document}

\begin{abstract}
\noindent The purpose of this paper is to introduce and study the  notions of  $f$-rack and $f$-quandle which are obtained by twisting the usual equational identities by a  map.  We provide some key constructions, examples and classification of low order $f$-quandles.  Moreover, we define modules over $f$-racks, discuss 
extensions and define a cohomology theory for $f$-quandles and give examples.
\end{abstract}

\maketitle


\tableofcontents

%

\section{Introduction}

\noindent  Shelves, racks and quandles are algebraic structures whose axioms come from Reidemeister moves in knot theory.  The earliest known work on racks is contained in the 1959 correspondence between John Conway and Gavin Wraith \cite{CW} who studied racks in the context of  the conjugation operation in a group.  Around 1982,  the notion of a quandle was introduced independently by Joyce \cite{Joyce} and Matveev \cite{Matveev}.  They used it to construct representations of  braid groups.  Joyce and Matveev associated to each knot a quandle that determines the knot up to isotopy and mirror image.  Since then quandles and racks have been investigated by topologists  in order to construct knot and link invariants, see \cite{Carter-Crans-Elhamdadi-Saito,Carter-Elhamdadi-Grana-Saito, CES, CENS, FR, EN} for more details.  Quandles also can be investigated on their own right as non-associative algebraic structures.\\
A Hom-algebra structure is a multiplication on a vector space where the structure is twisted by some map.  Thus any algebraic structure defined by axioms can be twisted to give a new generalization of the structure.  The former structure can be then seen as a specialization of the later when the map is the identity.  An example is the Hom-Lie algebras which were introduced in \cites{HLS, LS} with the goal of studying the deformations of Witt algebra, which is the complex Lie algebra of derivations of the Laurent polynomials in one variable, and the deformation of the Virasoro algebra, a one-dimensional central extension of Witt algebra.   Since then Hom-structures in different settings (algebras, coalgebras, Hopf algebras, Leibniz algebras, $n$-ary algebras, etc.) were investigated by many authors (see for example \cites{HLS,LS,MS4,MakhloufSilv,Yau1,Yau2, Yau3}).  In this article, we apply this principle of twisting the algebraic equations to the context of quandles to obtain a notion of $f$-quandle.  We study the properties of $f$-racks, $f$-quandles, their classification, extensions, modules and their cohomology.

\noindent The paper is organized as follows.  In Section 2, we define $f$-quandles and give some examples.  Classification of $f$-quandles of low dimension are provided in Section 3.  In Section 4, dynamical extensions of $f$-quandles are defined giving rise to a notion of $f$-quandle $2$-cocycle.  Explicit formulas relating group $2$-cocycles to $f$-quandle $2$-cocycles, when the $f$-quandle are constructed from groups are given.  In Section 5 we develop a cohomology theory and in the last remark, we discuss some natural questions on connections with Knot Theory. \\

\section{$f$-quandles}
We start by reviewing the notion of quandle and give some examples.
\noindent A {{\it quandle}} $X$ is a non-empty set with a binary operation $(a, b) \mapsto a  \rt  b$
satisfying the following axioms:
\begin{eqnarray*}
\label{axiom1}
({\rm I}) & & \mbox{\rm  For any $a \in X$,
$a \rt  a =a$.}  \\
({\rm II}) & &\mbox{\rm For any $b,c \in X$, there is a unique $a \in X$ such that
$ a \rt b=c$.} \label{axiom2} \\
\label{axiom3}
({\rm III}) & & \mbox{\rm For any $a,b,c \in X$, we have
$ (a \rt  b) \rt  c=(a \rt  c) \rt (b \rt c). $}
\end{eqnarray*}

\noindent Typical examples of quandles include the following:
  \begin{itemize}
 \setlength{\itemsep}{-3pt}

\item
A group $X=G$ with
$n$-fold conjugation
as the quandle operation: $a \rt  b=b^{-n} a b^n$.\\

\item
Let $n$ be a positive integer.
For
$a, b \in \Z_n$ (integers modulo $n$),
define
$a  \rt  b \equiv 2b-a \pmod{n}$.
Then the operation $ \rt $ defines a quandle
structure  called the {\it dihedral quandle},
  $R_n$.
This set can be identified with  the
set of reflections of a regular $n$-gon
  with conjugation
as the quandle operation.\\
\item
Any ${\Z }[T, T^{-1}]$-module $M$
is a quandle with
$a \rt b=Ta+(1-T)b$, $a,b \in M$, called an {\it  Alexander  quandle}.
  \end{itemize}

 \noindent
  Let $X$ be a quandle.  The {\it right translation}   $\mathcal{R}_a:X \rightarrow  X$, by $a \in X$, is defined by ${\mathcal{R}}_a(x) = x \rt a$ for $x \in X$.  Then by axiom (II), the map  ${\mathcal R}_a$ is a bijection of $X$.  It is seen that the operation $\overline{ \rt }$ on $X$ defined by $x\  \overline{ \rt }\  a= {\mathcal R}_a^{-1} (x) $ is a quandle operation satisfying $(x \rt a)\; \overline{ \rt } \; a=a=(x\;  \overline{ \rt } \;a) \rt a$.  The set $(X,  \overline{ \rt }) $ is called the {\it dual} quandle of $(X,  \rt )$.

\subsection{Definitions and Properties}

\begin{definition}\label{twistdef}
A {\it $f$-shelf} is a triple $(X, *, f)$ in which $X$ is a set, $*$ is a binary operation on $X$, and $f \colon X \to X$ is a map such that, for any $x,y,z \in X$, the identity
\begin{equation} \label{twistedCon-I}
(x *y) * f(z) =  (x * z) * (y * z)
 \end{equation}
holds.  A {\it $f$-rack} is a $f$-shelf such that, for any $x,y \in X$, there exists a unique $z \in X$ such that
\begin{equation} \label{twistedCon-II}
z* y = f(x).
\end{equation}
A {\it $f$-quandle} is a $f$-rack such that, for each $x \in X$, the identity 
\begin{equation} \label{twistedCon-III}
x* x =f(x)
\end{equation}
holds.\\
A \textit{ $f$-crossed set} is a $f$-quandle $(X, *, f)$ such that  $f: X \to X$ satisfies $x * y =f(x) $ whenever   $ y* x = f(y)$ for any $x, y \in X$.
\end{definition}

\begin{remark}
Using the right translation $R_a: X \to X$ defined as $R_a(x) = a * x$,  the identity (\ref{twistedCon-I})  can be written as $R_{f(z)}(x*y) = R_z(x) * R_z(y) $ or $R_{f(z)} \circ R_y = R_{R_z(y)}\circ R_z$ for any $x, y , z \in X$.\\
The extra condition in the $f$-crossed definition means that, for any $x,y\in X$,  $R_x(y)=f(y)$ is equivalent to  $R_y(x)=f(x)$.
\end{remark}


\begin{definition}
Let $(X_1, *_1, f_1)$ and $(X_2, *_2, f_2)$ be two $f$-racks (resp. $f$-quandles). A map $\phi : X_1 \to X_2$ is a $f$-rack (resp. $f$-quandle) morphism if it satisfies $\phi (a *_1 b) = \phi(a) *_2 \phi(b)$ and $\phi \circ f_1 = f_2 \circ \phi$.
\end{definition}

\begin{remark}
A category of $f$-quandle is a category whose objects are tuples $(A, *, f)$ which are $f$-quandles and morphism are $f$-quandle morphisms.
\end{remark}

\begin{example}\label{Examples}
Examples of $f$-quandles include the following:

  \begin{enumerate}
  \setlength{\itemsep}{-3pt}

\item
 Given any set  $X$ and map $f: X \to X$, then the operation  $ x * y = f(x)$ for any $ x, y \in X$ gives a $f$-quandle. We call this a \textit{trivial} $f$-quandle structure on $X$.\\

\item
For any group $G$ and any group endomorphism $f$ of $G$, the operation $x *  y=y^{-1} x f(y) $ defines a $f$-quandle structure on $G$.\\

\item

Consider the Dihedral quandle $R_n$, where $ n \geq 2$, and let $f$ be an automorphism of $R_n$.  Then $f$ is given by $f(x)=ax+b$, for some invertible element $a \in \mathbb{Z}_n $ and some $ b \in \mathbb{Z}_n$ \cite{EMR}.  The binary operation
$x  *  y =f(2y-x)= 2ay- ax +b  \pmod{n}$ gives a $f$-quandle structure called the  {\it  $f$-Dihedral quandle}.\\
\item
Any ${\Z }[T^{\pm 1}, S]$-module $M$
is a $f$-quandle with
$x * y=Tx+Sy$, $x,y \in M$ with $ TS = ST$ and $f(x)=(S+T)x$, called an {\it  Alexander  $f$-quandle}.
  \end{enumerate}
\end{example}

\begin{remark}
Axioms (\ref{twistedCon-I}) and (\ref{twistedCon-III}) of Definition \ref{twistdef} give the following equation,  $(x * y) * ( z * z)  =  (x*z) * (y * z) $.  We note that the two medial terms in this equation are swapped (resembling the mediality condition of a quandle).  Note also that the mediality in the general context may not be satisfied for $f$-quandles. For example one can check that the $f$-quandle given in item (2) of Example \ref{Examples} is not medial.
\end{remark}

\begin{proposition}
If $(X, *, f)$ is a $f$-quandle, then right multiplication  $R_{y}(x)=x*y$ is a bijection.
\end{proposition}

\begin{proof}
Assume $ a*b = c*b$ where $ a \neq c$. Then $ f(a) * f(b) = f(a * b) = f(c) * f(b) \implies f(a) = f(c) $. Consider $(a*c) * f(b)$.
 \begin{align*}
		(a*c) * f(b)
			&= (a*b) * (c*b)    \text{ by axiom (\ref{twistedCon-I}) of Definition \ref{twistdef}}\\
			&= f(c*b)          \text{ since $a*b = c*b$}\\
			&= f(c) * f(b).			
	\end{align*}
So $ a*c = f(c) = c*c \implies a=c$ which is a contradiction.
\end{proof}
In the following, we provide a key construction providing a new $f$-quandle from a $f$-quandle and a $f$-quandle morphism. This gives a way to produce examples.
\begin{proposition}\label{twistprop}
Let $(X,*,f)$ be a finite $f$-quandle and $\phi:X\rightarrow X$ be a $f$-quandle morphism: ($\phi(x*y)= \phi(x)*\phi(y)$).
Then $(X,*_{\phi},f_{\phi})$ is a $f$-quandle with $a*_{\phi}b=\phi(a*b)$ and $f_{\phi}(a)=\phi(f(a))$ if and only if $\phi$ is an automorphism. Moreover, if $\phi$ is an element of the centralizer of the automorphism group of $X$, then the associated $f$-quandle has the same $f$ map.
\end{proposition}
\noindent We will refer to $(X,*_{\phi},f_{\phi})$ as a { \it twist} of $(X,*,f)$.

\begin{proof}$ \Leftarrow) $  Assume that $\phi$ is an automorphim.  Let $a,b,c\in X$
\begin{eqnarray*}
(a *_\phi b)*_\phi f_\phi (c)&=& (a *_\phi b) *_\phi \phi(f(c)) = \phi(a* b)*_\phi \phi(f(c))\\
&=&\phi [(a*b) *_\phi f(c)]
= \phi^2((a* c)* (b* c))\\
&=& \phi[\phi(a* c)* \phi((b* c))] = \phi[(a *_\phi c) * ( b *_\phi c)]\\
&=& (a*_\phi c)*_\phi (b *_\phi c).
\end{eqnarray*}
Obviously $a*  a=f(a)$ implies $f_\phi(a) = \phi(f(a)) = \phi(a*a) = a *_\phi a$.  Now given $a,b\in X$ there exists $c\in X$ such that $c * b=f(a)$, therefore  $c *_\phi  b=\phi(c* b)=\phi(f(a))=f_\phi (a)$.\\

$\Rightarrow)$ Assume that $(X,*_{\phi},f_{\phi})$ is a $f$-quandle with $a*_{\phi}b=\phi(a*b)$ and $f_{\phi}(a)=\phi(f(a))$.\\
We need to prove that $\phi(a) = \phi(b)$ implies $a = b$.\\
Assume $ \phi(a) = \phi(b)$. Let $c, d$ be such that $ a * c = f(d)$.
\begin{eqnarray*}
&& \phi(a * c) = \phi(f(d))  \implies \phi(a) * \phi(c) = \phi(f(d)) \implies a *_\phi c = f_\phi (d).\\
&& \phi(b) * \phi(c) = f_\phi(d) \implies b*_\phi c = f_\phi (d).
\end{eqnarray*}
Hence $a = b$ . \\
Since $X$ is finite and $\phi$ is an injective homomorphism then $\phi$ is an automorphism.

\end{proof}

\begin{remark} Notice that a quandle (resp. rack,  shelf) may be viewed as  a $f$-quandle (resp. $f$-rack, $f$-shelf)  for which the structure map  $f$ is the identity map.  \end{remark}

\begin{corollary}
\noindent In the particular case where $f=id_x$, the proposition shows that any usual quandle along with an appropriate morphism gives rise to a $f$-quandle.
 \end{corollary}

\begin{example}\label{S31}
Let $G$ be a non-abelian group and  $f$ a group automorphism. Then one gets an example of a $f$-quandle\ by $x*y=f(y)^{-1}f(x)f(y)$.  For example,  in the case of the symmetric group on three letters $G=S_3 = <s, t: s^2 = t^3 = e, ts = st^2>$ and $f$ be a group automorphism that maps $s \to st$, $ t \to t^2$. Then one gets an example of a $f$-quandle by $x*y = f^{-1}(y)f(x)f(y)$ given by the following Table:\\
$\begin{array}{c|ccccccc}
* & e & s & t & t^2 & st & st^2 \\
\hline
e & e & e & e & e & e & e \\
s & st & st & st^2 & s & st^2 & s\\
t & t^2 & t & t^2 & t^2 & t & t\\
t^2 & t & t^2 & t & t & t^2 & t^2\\
st & s & st^2 & st & st^2 & s & st\\
st^2 & st^2 & s & s & st & st & st^2
\end{array}$
\end{example}

\begin{remark}
If we consider in the previous example the operation defined as $x * y = y^{-1}xf(y)$, then we obtain an isomorphic $f$-quandle.
\end{remark}

\begin{example}
Recall from \cite{EMR} that any automorphism of the dihedral quandle $\mathbb{Z}_n$ is of the form $f_{a,b}(x) = ax + b$. Using the previous proposition we recover the $f$-Dihedral quandle in item 3 of Example \ref{Examples}.
\end{example}

\subsection{Constructions}
We provide some key constructions. In the following, we associate to a $f$-rack $(X,*,f)$ a $f$-crossed set structure on $S_X$, the set of permutations defined by the right multiplications. 

\begin{proposition} 
	Let $(X,*,f)$ be a $f$-rack, then $S_X$, the set of permutations defined by right multiplication $R_x$, form a $f$-crossed set, under the natural operation $R_x *_R R_y = R_{x*y}$ and the morphism $f_R(R_x)=R_{f(x)}$.
\end{proposition}

\begin{proof}
	First note that by the second axiom of $f$-quandles $R_x \in S_X$, and consider the $f$-quandle formed from $S_X$ by defining $a * b = f(b) a b^{-1}$ for $a,b \in S_X$. By the first axiom of $f$ quandles we have $R_{f(z)}R_y = R_{y*z}R_{z}$. Thus, $R_y *_R R_z = R_{y*z} = R_{f(z)} R_y R_z^{-1}$. The closure of the operation on this subset is clear from its definition, thus $(R_X,*_R,f_R)$ is a $f$-subquandle. Further if $R_{x*y}=R_{f(x)}$ then
	$$R_{y*x}=R_{f(x)} R_y R_x^{-1}=R_{f(y)} R_z R_y^{-1} R_y R_z^{-1}=R_{f(y)}.$$
	Thus it is a $f$-crossed set.
\end{proof}

We define in the following the concept of Enveloping groups of $f$-racks.

\begin{definition}
Let $(X, *, f) $ be a $f$-rack. Then there is a natural map $\iota$  mapping $X$ to a group, called the enveloping group of $f$-rack of $X$ , and defined  as $G_X = F(X) / <x *y = f(y)xy^{-1} , x, y \in X>$, where $F(X)$ denotes the free group generated by $X$.
\end{definition}
\noindent In the following, we discuss a functoriality property between $f$-racks and groups, see \cite{And-Grana,FR} for the classical case.
\begin{proposition}
Let $(X,*,f)$ be a $f$-rack and $G$ be a group. Given any $f$-rack homomorphism  $\varphi:X\to G_{conj}$, where $G_{conj}$ is a group together with a $f$-rack structure along a group homomorphism $g$, that is the multiplication is defined as $a*_Gb=g(b)ab^{-1}$. Then, there exists a unique group homomorphism $\widetilde{\varphi}:G_X\to G$ which makes the following diagram commutative
\[
\xymatrix{
(X,*,f) \ar[rr]^{\iota} \ar[d]_{\varphi} &&G_X \ar[d]^{\widetilde{\varphi}}  \\
(G_{conj},*_G,g) \ar[rr]_{id} && G
}
\]
\end{proposition}
\begin{proof}
Let $\bar{\varphi} : F(X)\to G$  be a $f$-rack homomorphism extension of $\varphi$ to the free group $F(X)$. Then $\bar{\varphi}(yx^{-1}f(y)^{-1}(x*y)) = 1$. Indeed,\\
\begin{eqnarray*}
\bar{\varphi}(yx^{-1}f(y)^{-1}(x*y)) &=& \varphi(y)\varphi(x^{-1})\varphi(f(y)^{-1})\varphi(x*y)\\
&=& \varphi(y)\varphi(x)^{-1}\varphi(f(y))^{-1}(\varphi(x)*_G\varphi(y))\\
&=&  \varphi(y)\varphi(x)^{-1}\varphi(f(y))^{-1}g(\varphi(y))\varphi(x)\varphi(y)^{-1}\\
&=& 1.
\end{eqnarray*}
Since $\varphi \circ f = g \circ \varphi$ ($f$-rack morphism). It follows that $\varphi$ factors through a unique homomorphism $\tilde{\varphi} : G_X \to G$ of groups. The commutativity of the diagram is straightforward.
\end{proof}

 \begin{remark}
The functor $(X, *, f) \to G_X$ is left adjoint to the forgetful functor $G \to G_{conj}$ from the category of groups to that of $f$-racks. That is,
\begin{center} 
$Hom_{groups}(G_X, G) \simeq Hom_{ f - racks}(X, G_{conj})$ 
\end{center}by natural isomorphism.
\end{remark}
	
Now, we discuss relationships between $f$-quandles and $f$-crossed sets.
We extend to the $f$ case  the construction, provided in \cite{And-Grana}, of a functor $Q$ from the category of finite  quandles with bijective maps to crossed sets. We assign to a $f$-quandle $X$ a $f$-crossed set $Q(X)$ with the expected universal property: any morphism of $f$-quandles factorized through $Q(X)$ whenever Y is a $f$-crossed set.\\
Let  $\phi : X \to Y$ be a morphism on a set $X$ with the equivalence relation $\sim$ defined by $ x \sim x^{'}$ if there exists  $ y$ such that  $x * y = f(x^{'})$ and $y * x = f(y)$. \label{I} \\Then $\sim$ coincides with the identity relation if and only if $X$ is a crossed set.\\ Let $X_1=X/ \sim $ , we see that $X_1$ inherits the structure of a $f$-quandle. Suppose $x, y , x^{'}$ as above, then $ f(y) * f(x^{'}) = (y * y) * (x*y) = (y*x) * f( y) = f(y)*f(y) =f(y*y)=f(f(y))$. Using invertibility of $f$, we have  that $y * x^{'} = f(y)$. \\ Moreover, for $z \in X$,  \begin{eqnarray*}\lefteqn{  (z * x^{'}) * f(y)}\\ && = (z*x^{'}) * (y *x^{'}) = (z*y) * f(x^{'})   \\ && =
  (z*y) * (x*y) = (z*x)*f(y).
\end{eqnarray*}
Hence, we see that $R_x = R_{x^{'}}$ and then $R_x = R_{x^{''}}$ for any $ x^{''} \sim x$. Therefore, we extend the operation to $ * : X \times (X/\sim) \to X $. Also, for any $z \in X$, $$(x * z) * (y*z) = (x*y)*f(z) = f(x^{'}) * f(z) = f(x^{'}*z)$$ and $$(y*z)*(x*z) = (y*x)*f(z) = f(y) * f(z) = f(y * z).$$ Then $(x*z) \sim (x^{'}*z)$.  So we can consider $* : (X/\sim) \times (X/\sim) \to (X/\sim)$. If $X_1$ is not a $f$-crossed set then by taking $X_2 = X_1/\sim$ and proceed the same way until arrive to a $f$-crossed set. Hence the functoriality and universal property.

\section{Classification of $f$-quandles}

In this section we give a classification of $f$-quandles of low order and discuss those $f$-quandles which can be obtained from quandles via Proposition \ref{twistprop}. First we will introduce the notion of $f$-isomorphisms and the resulting $f$-isomorphism classes.

\begin{definition}
\
Let $(X,*,f)$ and $(X_0,*_0,f_0)$ be $f$-quandles. If there exists an automorphism $\phi: X \to X$ and corresponding twisting of $(X,*,f)$, denoted $(X,*_\phi,f_\phi)$, that is isomorphic to $(X_0,*_0,f_0)$, then $(X,*,f)$ and $(X_0,*_0,f_0)$ are said to be twisted-isomorphically equivalent.
\end{definition}

\noindent It is easy to see that if $(X,*,f)$ is twisted-isomorphically equivalent to $(X_0,*_0,f_0)$ with twisting automorphism $\phi: X \to X$ and isomorphism $\psi: X \to X_0$, then there exists a $f$-quandle $(X_0,*_1,f_1)$ that is isomorphic to $(X,*,f)$ via the same isomorphism and $(X_0,*_1,f_1)$ is a twist of $(X_0,*_0,f_0)$ with the automorphism $\psi^{-1}\phi\psi$, thus the following  diagram commutes. 
\[
\xymatrix{
(X,*,f) \ar[rr]^{\phi} \ar[d]_{\psi} && (X, *_{\phi}, f_{\phi}) \ar[d]^{\psi}  \\
(X_0, *_1, f_1) \ar[rr]_{\psi^{-1}\phi \psi} && (X_0, *_0, f_0)
}
\]

From this it becomes clear that twisted-isomorphisms do define equivalence classes on the space of $f$-quandles.

\noindent This allows us to consider the classification of $f$-quandles in terms of twisted-isomorphically distinct classes. This method of classification brings us to the following remark.

\begin{remark}
Let $(X,*,f)$ be a $f$-quandle such that $f$ is a bijection, then there exists a quandle $(X,\rhd)$ such that $a*b=f(a \rhd b)$. That is, $(X,*,f)$ is a twist of $(X,\rhd,id_X)$
\end{remark}

\noindent As $f: X \to X$ is a automorphism on $(X,*,f)$, so is $f^{-1}$, and twisting $(X,*,f)$ by $f^{-1}$ gives the $f$-quandle $(X,*_2, id_X)$ which is a quandle. As such the classification of $f$-quandles with bijective maps can be found by extending the isomorphic classes of quandles to twisted isomorphic classes. As this classification is far more extensive and readily accessible, we will limit our further discussion of classification to those $f$-quandles with non-bijective $f$-maps.

\noindent Before continuing we will note two other special classes of $f$-quandle. Having classified $f$-quandles such that $f$ is a bijection, we will note an interesting property of those $f$-quandles with constant $f$-maps.

\noindent \begin{remark} If $(X, *, f)$ is a $f$-quandle, such that there exists  $a \in X$ such that $f(b) = a $ for all $b \in X$, then $(X, *, f)$ is a Latin square.\end{remark}

\begin{proof} Assume $ f(a) = 1$  for all $a$. Assume also $ a* b = a * c$  with $b \neq c$. Then \begin{align*}
		(a * b) * f(b)
			&= ( a * c) * f(b)\\
			&= ( a * b ) * ( a * c )\\
			&= f( a * c)
	\end{align*}
As $f(b) = f( a * c) = 1$ then $ ( a * b) * 1 = 1 $ by axiom II, we have $(a * b) = 1 \implies a = b$.\\
Switching $ b$ and $c$ in the above gives $ a = c \implies b = c$ which is a contradiction. Therefore left multiplication is a bijection.
By definition, $f$-quandles are right bijection, hence the proof. \quad \end{proof}

\noindent The other special class is note worthy mainly due to the ease with which they can be described. Affine $f$-quandles are those defined over $\mathbb{Z}_n$ with an operation of the form $a*b \equiv sa+tb+r$ where $s,t,r \in \mathbb{Z}_n$.

\noindent Now, we present the results of a simple program used to generate all $f$-quandles of order less than five. \\
For the following results, the structure map $f$ is defined as $f(a)=a*a$.
 \begin{theorem}
There exists one $f$-isomorphic class of order 2 that does not contain a quandle. It is affine, defined over $\mathbb{Z}_2$ with the operation $a *_2 b \equiv a+b$.
\end{theorem}

\begin{remark}
A direct calculation shows that the classification of order 2 $f$-quandles reduces to the following four cases:\\
\begin{tabular}{llrr}
\begin{tabular}{|c|c|}
\hline
1 & 2  \\
\hline
2 & 1  \\
\hline
\end{tabular}
 &
\begin{tabular}{|c|c|}
\hline
1 & 1 \\
\hline
2 & 2 \\
\hline
\end{tabular}
&
\begin{tabular}{|c|c|}
\hline
2 & 2 \\
\hline
1 & 1 \\
\hline
\end{tabular}
&
\begin{tabular}{|c|c|}
\hline
2 & 1 \\
\hline
1 & 2 \\
\hline
\end{tabular}
\end{tabular}\\

The first is the twist of the  fourth and the second is the twist of the third as covered by Proposition 2.14. \end{remark}

\noindent It is interesting to note that the $f$-quandle above is commutative, and in fact it is easy to see that all cyclic groups are in fact $f$-quandles.

\begin{theorem}
There exists no twisted isomorphic class of order 3 that do not contain a quandle. 
\end{theorem}

\noindent It is in our classification of order 4 where we encounter the smallest examples of $f$-quandles that are neither twists of ordinary quandles nor cyclic groups.

\begin{theorem}
There exist 4 $f$-isomorphic classes of order 4 what do not contain a quandle. Two of which are isomorphic to the additive groups $\mathbb{Z}_4$ and $\mathbb{Z}_2 \times \mathbb{Z}_2$. Representatives of the the remaining two classes are given by the following Cayley tables.\\

\begin{tabular}{lr}
\begin{tabular}{|c|c|c|c|}
\hline
1 & 2 & 4 & 3 \\
\hline
2 & 1 & 3 & 4 \\
\hline
3  & 4 & 1 &  2 \\
\hline
4 & 3 & 2 & 1\\
\hline
\end{tabular}
 &
\begin{tabular}{|c|c|c|c|}
\hline
1 & 2 & 4 & 3 \\
\hline
2 & 1 & 3  & 4 \\
\hline
3 & 4 & 2 & 1 \\
\hline
4 & 3 &  1 & 2\\
\hline
\end{tabular}
\end{tabular}

\end{theorem}

\section{Extensions of $f$-quandles and Modules}

\noindent
In this section we investigate extensions of $f$-quandles. We define generalized $f$-quandle $2$-cocycles and give examples.  We give an explicit formula relating group $2$-cocycles to $f$-quandle $2$-cocycles, when the $f$-quandle is constructed from a group.
\subsection{Extensions with dynamical cocycles and Extensions with constant cocycles}
\noindent \begin{proposition}Let $( X, *, f) $ be a $f$-quandle and $A$ be a non-empty set. Let $\alpha: X \times X \to \text{Fun}(A \times A, A)$ be a function and $g: A \to A$ a map.
Then,  
$X \times  A$ is a $f$-quandle by the operation $(x, a) * (y, b) = ( x*y, \alpha_{x,y}(a, b))$, where $x * y$ denotes the $f$-quandle product in $X$, if and only if $\alpha$ satisfies the following conditions:
\begin{enumerate}
\item $\alpha_{x,x}(a, a) = f(a)$ for all $x \in X$ and $ a \in A$;
\item $\alpha_{x, y}(-, b) : A \to A$ is a bijection for all $ x, y \in X$ and for all $ b \in A$;
\item $\alpha_{x*y, f(z)}(\alpha_{x, y}(a, b), g(c)) = \alpha_{x*z, y*z}(\alpha_{x, z}(a, c), \alpha_{y, z}(b, c))$ for all $ x, y, z \in X$ and $ a, b, c \in A$. \label{cocycle-3}
\end{enumerate}
If $(X, *, f)$ is a $f$-crossed set, then $X \times A$ is a $f$-crossed set if and only if further $\alpha_{x,y}(a,b) = f(b)$ whenever $y*x = f(y)$ and $\alpha_{y,x}(b,a) = f(a)$.\\

\noindent Such function $\alpha$ is called a \textit{dynamical $f$-quandle cocycle} or \textit{dynamical $f$-rack cocycle} (when it satisfies above conditions).

\end{proposition}

\noindent The $f$-quandle constructed above is denoted by $X \times_{\alpha} A$, and it  is called \textit{extension} of $X$ by a dynamical cocycle $\alpha$.  The construction is general, as Andruskiewitch and Gra$\tilde{n}$a showed in \cite{And-Grana}.\\

\noindent Assume $(X, *, f)$ is a $f$-quandle and $\alpha$  be a dynamical cocycle. For $x \in X$, define $a *_x b := \alpha_{x,x}(a, b)$. Then it is easy to see that $(A, *_x)$ is a $f$-quandle for all $x \in X$.

\begin{remark}
When $x = y$ on (\ref{cocycle-3})  above, we get \\
\noindent $\alpha_{f(x), f(z)}(\alpha_{x, x}(a, b), g(c)) = \alpha_{x*z, x*z}(\alpha_{x, z}(a, c), \alpha_{x, z}(b, c)) = \alpha_{x,z}(a,c) *_{x*z} \alpha_{x,z}(b,c)$ \\for all  $ a, b, c \in A$. When $f, g = id$, then it reduces to classical case where $\alpha_{x,z}(a):(A, *_x)\to(A, *_{x*z})$ is an isomorphism.
\end{remark}

Now, we discuss Extensions with constant cocycles.
Let $(X, *, f)$ be a $f$-rack and $\lambda: X \times X \to S_A$ where $S_A$ is group of permutations of $X$.\\
If $\lambda_{x*y, f(z)} \lambda_{x, y} = \lambda_{x*z, y*z}\lambda_{x, z}$ we say $\lambda$ is a \textit{constant $f$-rack cocycle}.\\
If $(X, *, f)$ is a $f$-quandle and further satisfies $\lambda_{x,x} = id$ for all $x \in X$ , then we say $\lambda$ is a \textit{constant $f$-quandle cocycle}.\\
If $(X, *, f)$ is a $f$-crossed set, we say that $\lambda$ is a \textit{constant $f$-crossed set cocycle} if it further satisfies $\lambda_{x,y} = id $ whenever $y*x = f(y) $ and $\lambda_{y,x}(b) = f(b)$ for some  $b \in A$.
\subsection{Modules over $f$-rack}
\begin{definition}
Let $(X, *, f)$ be a $f$-rack, $A$ be an abelian group, and $g : X \to X$ be a homomorphism. A structure of $X$-module on $A$ consists of a family of automorphisms  $(\eta_{ij})_{i, j \in X}$ and a  family of  endmorphisms $(\tau_{ij})_{i, j \in X}$ of  $A$ satisfying the following conditions:\\
\begin{eqnarray}
&& \eta_{x*y, f(z)} \eta_{x, y} = \eta_{x*z, y*z} \eta_{x, z} \label{mod-con-I}\\
&& \eta_{x*y, f(z)} \tau_{x, y} = \tau_{x*z, y*z }\eta_{y, z} \label{mod-con-II}\\
&& \tau_{x*y, f(z)}g = \eta_{x*z, y*z} \tau_{x, z} + \tau_{x*z, y*z} \tau_{y, z} \label{mod-con-III}
\end{eqnarray}
\end{definition}

\begin{remark}
If $X$ is a $f$-quandle, a $f$-quandle structure of $X$-module on $A$ is a structure of an $X$-module further satisfies $ \tau_{f(x),f(x)} g = (\eta_{f(x),f(x)} + \tau_{f(x),f(x)} ) \tau_{x,x}$.\\
Furthermore, if $f, g =id$ maps, then it satisfies $ \eta_{x,x} + \tau_{x,x} = id $.
\end{remark}

\begin{remark}
When $ x=y$ in \eqref{mod-con-I}, we get $\eta_{f(x),f(z)} \eta_{x,x} = \eta_{x*z, x*z} \eta_{x,z}.$
\end{remark}


\begin{example} Let $A$ be a non-empty set and
$(X, f)$ be a $f$-quandle, and $\kappa$ be a generalized $2$-cocycle. For $a, b \in A$, let\\
\begin{center}
$ \alpha_{x, y}(a, b) = \eta_{x, y}(a) + \tau_{x, y}(b) + \kappa_{x, y}$.
\end{center}
Then, it can be verified directly that $\alpha$ is a dynamical cocycle and the following relations hold:
\begin{eqnarray}
 &&\eta_{x*y, f(z)} \eta_{x, y} = \eta_{x*z, y*z} \eta_{x, z}\\
 &&\eta_{x*y, f(z)} \tau_{x, y} = \tau_{x*z, y*z }\eta_{y, z}\\
 &&\tau_{x*y, f(z)}g = \eta_{x*z, y*z} \tau_{x, z} + \tau_{x*z, y*z} \tau_{y, z}\\
 &&\eta_{x*y, f(z)} \kappa_{x, y} + \kappa_{x*y, f(z)} = \eta_{x*z, y*z} \kappa_{x, z} + \tau_{x*z, y*z} \kappa_{y, z} + \kappa_{x*z, y*z}.\label{$2$-cocycle}
 \end{eqnarray}
\end{example}

\begin{definition}When $\kappa$ further satisfies $\kappa_{z,z} =0 $ in (\ref{$2$-cocycle}) for any $z \in X$, we call it a \textit{generalized $f$-quandle $2$-cocycle}.\end{definition}


\begin{example} \label{etaidtauzero}
Let $(X,*_f, f)$ be a $f$-quandle and $A$ be an abelian group.\\ Set $\eta_{x, y} = id, \tau_{x, y} = 0, \kappa_{x, y} = \phi(x, y) $.\\
Then $ \phi(x,y) + \phi(x *_f y, f(z)) = \phi(x,z) + \phi(x *_f z, y *_f z)$.
\end{example}

\begin{example} Let $\Gamma = \mathbb{Z}[T,S]$ denote the ring of Laurent polynomials. Then any $\Gamma$-module $M$ is a $\mathbb{Z}(X)$-module for any $f$-quandle $(X, f)$ by $\eta_{x, y}(a) = Ta$ and $\tau_{x, y}(b) = Sb$ for any $x, y \in X$.
\end{example}

\begin{example} For any $f$-quandle $(X, f)$, the group $G_X =< x \in X | x * y = y^{-1}xf(y) >$ with  $\eta_{x, y}(a) = ya$ and $\tau_{x, y}(b) = ( 1 - x*y)(b)$ where $x, y \in X$ and $a, b \in G$, is an $X$-module.
\end{example}

%
%

\begin{example}
Here we provide an example of a $f$-quandle module and explicit formula of the $f$-quandle $2$-cocycle obtained from a group $2$-cocycle.
Let $G$ be a group and let $ 0 \to A \to E \to G \to 1$ be a short exact sequence of groups where $E = A \rtimes_{\theta}   G$ by a group $2$-cocycle $\theta$ and $A$ is an Abelian group.\\
The multiplication rule in $E$ given by $(a, x) \cdot (b, y) = ( a + x \cdot b + \theta(x, y), xy)$, where $x \cdot b$ means the action of $A$ on $G$. Recall that the group $2$-cocycle condition is $\theta(x, y) + \theta(xy, z) = x\theta(y, z) + \theta(x, yz)$.  Now, let $X=G$ be a $f$-quandle with the operation $ x * y = y^{-1}xf(y)$ and let $g: A \rightarrow A$ be a map on $A$ so that we have a map $F:E \rightarrow E$  given by $F(a,x)=(g(a), f(x))$.  Therefore the group $E$ becomes a $f$ quandle with the operation $ (a,x) * (b,y) = (b,y)^{-1} (a,x) F(b,y)$.  Explicit computations give that
 $\eta_{x, y}(a) = y^{-1}a$, $\tau_{x, y}(b) = y^{-1}xf(b) - y^{-1}b$ and $\kappa_{x, y} = -\theta(y^{-1}, y) +  \theta(y^{-1},x) + \theta(y^{-1}x, g(y))$.

\end{example}

\section{Cohomology Theory of $f$-quandles}

In this section we first give the formula of the boundary map in the simplest case when $\eta_{x,y}$ is the identity map and $\tau$ is the zero map as in Example \ref{etaidtauzero} and then give the formula in the case when $\eta_{x,y}$ is the multiplication by $T$ and $\tau$ is the multiplication by $S$ as in item 4 of Example \ref{Examples}.\\

\noindent Let $(X, *, f)$ be a $f$-rack where $f: X \to X$ is a $f$-rack morphism. We will define the most generalized cohomology theories of $f$-racks as follows:\\
For a sequence of elements $(x_1, x_2, x_3, x_4,\dots, x_n) \in X^n$  define\\

$  [x_1, x_2, x_3, x_4,\dots, x_n] =  (( \dots (x_1 * x_2) * f(x_3)) * f^2(x_4))* \dots ) *  f^{n-2}(x_n)$.\\

\noindent Notice that for $ i <n$ we have \\

 $ [x_1, x_2, x_3, x_4,\dots, x_n] = [x_1,\dots, \hat{x}_i,\dots, x_n] * f^{i-2}[x_i,\dots, x_n]$\\

\noindent This relation is obtained by applying the first axiom of $f$-quandles $n-i$ times, first grouping the first $i-1$ terms together, then iterating this process, again grouping and iterating each.


Let $(X,*,f)$ be an $f$-quandle.  Consider the free left $\mathbb Z(X)$-module $C_n(X):= \mathbb Z(X)X^n $ with basis $X^n$.  For an abelian group $A$, denote $C^n(X,A):= Hom_{\mathbb Z(X)}(C_n(X) ,A)$. 
\noindent The following theorem provides cohomology complex for  $f$-quandle.

\begin{theorem} The following family of operators $ \delta^n : C^n(X) \to C^{n+1}(X)$ defines a cohomology complex. \\
\begin{eqnarray*}
\lefteqn{
\delta^n \phi (x_1, \dots, x_{n+1}) } \nonumber \\ && =
\sum_{i=2}^{n+1} (-1)^{i}    \eta_{[x_1, \dots, \hat{x}_i, \dots, x_{n+1}],f^{\{i-2\}}[x_i, \dots, x_{n+1}]} \phi(x_1, \dots, \hat{x}_i, \dots, x_{n+1})
\nonumber \\
&&
- \sum_{i=2}^{n+1} (-1)^{i}   \phi  (x_1 \ast x_i, x_2 \ast  x_i, \dots, x_{i-1}\ast  x_i,f( x_{i+1}), \dots, f(x_{n+1}))
\nonumber\\
&&
+ (-1)^{n+1}  \tau_{[x_1, x_3, \dots, x_{n+1}],[x_2, \dots, x_{n+1}]} \phi (x_2, \dots, x_{n+1}).
\end{eqnarray*}

\end{theorem}


\begin{proof}

\noindent To prove that $\delta^{n+1}\delta^n=0$, we will break the composition into pieces, using the linearity of $\eta$ and $\tau$.

\noindent First, we will show that the composition of the $i^{th}$ term of the first summand of $\delta^n$ with the $j^{th}$ term of the first summand of $\delta^{n+1}$ cancels with the $(j+1)^{th}$ term of the first summand of $\delta^n$ with the $i^{th}$ term of the first summand of $\delta^{n+1}$ for $i\leq j$. As the sign of these terms are opposite, we need only to show that the compositions are equal up to their sign.

Now, we can see that the composition of the $i^{th}$ term of the first summand of $\delta^n$ with the $j^{th}$ term of the first summand of $\delta^{n+1}$ can be rewritten as follows:
\begin{eqnarray*}
\lefteqn{
\eta_{[x_1,...,\hat{x}_i,...,x_{n+1}],f^{i-1}[x_i,...,x_{n+1}]} \eta_{[x_1,...,\hat{x}_i,...,\hat{x}_{j+1},...,x_{n+1}],f^j[x_{j+1},...,x_{n+1}]}} \nonumber \\ && =
\eta_{[x_1,...,\hat{x}_i,...,\hat{x}_{j+1},...,x_{n+1}]*f^{j-1}[x_j,...,x_{n+1}], f^{i-1}[x_i,...,\hat{x}_{j+1},...,x_{n+1}]*f^{i-1}f^{j-i+1}[x_{j+1},...,x_{n+1}]} \nonumber \\ &&
\eta_{[x_1,...,\hat{x}_i,...,\hat{x}_{j+1},...,x_{n+1}],f^{j-1}[x_{j+1},...,x_{n+1}]} \\ &&=
\eta_{[x_1,...,\hat{x}_i,...,\hat{x}_{j+1},...,x_{n+1}]*f^{i-1}[x_i,...,\hat{x}_{j+1},...,x_{n+1}],f^j[x_{j+1},...,x_{n+1}]}\\&&
\eta_{[x_1,...,\hat{x}_i,...,\hat{x}_{j+1},...,x_{n+1}],f^{i-1}[x_i,...,\hat{x}_{j+1},...,x_{n+1}]}\\ &&=
\eta_{[x_1,...,\hat{x}_{j+1},...,x_{n+1}],f^{j}[x_{j+1},...,x_{n+1}]}\eta_{[x_1,...,\hat{x}_i,...,\hat{x}_{j+1},...,x_{n+1}],f^{i-1}[x_{i},...,\hat{x}_{j+1},...,x_{n+1}]},
\end{eqnarray*}
which is precisely the $(j+1)^{th}$ term of the first summand of $\delta^n$ with the $i^{th}$ term of the first summand of $\delta^{n+1}$.

Similar manipulations show that the composition of $\tau$ from $\delta^n$ with the $i^{th}$ term of the first sum of $\delta^{n+1}$ cancels with the composition of the $(i+1)^{th}$ term of the first sum of $\delta^n$ with $\tau$ from $\delta^{n+1}$ (with the same relation holding for the compositions of $\tau$ and the second summands), the composition of the $i^{th}$ term of the second summand of $\delta^n$ with the $j^{th}$ term of the second summand of $\delta^{n+1}$ cancels with the $(j+1)^{th}$ term of the second summand of $\delta^n$ with the $i^{th}$ term of the second summand of $\delta^{n+1}$, the composition of the $i^{th}$ term of the second summand of $\delta^n$ with the $j^{th}$ term of the first summand of $\delta^{n+1}$ cancels with the $(j+1)^{th}$ term of the first summand of $\delta^n$ with the $i^{th}$ term of the second summand of $\delta^{n+1}$ for $i\leq j$. For the sake of brevity we will omit showing these manipulations.\\

\noindent All these relations leave three remaining terms, which cancel via the third axiom in Definition 4.1.
\end{proof}
\noindent The Table below presents all the above relations in an easier to read manner. In the table $\eta_i$ represents the $i^{th}$ summand of the first sum, $\circ_i$ represents the $i^{th}$ summand of the second sum, with order of composition determining its origin in $\delta^n$ or $\delta^{n+1}$.
\begin{center}
\begin{tabular}{ c c c }
$\eta_i \eta_j$ & $=$ & $\eta_{j+1}\eta_i$ \\
$\eta_i \circ_j$ & $=$ & $\circ_{j+1} \eta_i$ \\
$\eta_i \tau$ & $=$ & $\tau \eta_{i+1}$ \\
$\tau \circ_i $ & $=$ &  $ \circ_{i+1} \tau $\\
$\circ_i \circ_j$ & $=$ & $\circ_{j+1} \circ_i$ \\
\end{tabular}
\end{center}

\begin{example} Let $(X,*,f)$ be a $f$-quandle and consider $C_n^{\rm R}(X)$ to be the free
abelian group generated by
$n$-tuples $(x_1, \dots, x_n)$ of elements of a quandle $X$. When $\eta$ is the identity map and $\tau$ is the zero map, we obtain 
$\delta_{n}: C_{n}^{\rm R}(X) \to C_{n-1}^{\rm R}(X)$ by

\begin{eqnarray*}
\lefteqn{
\delta_{n} \phi (x_1, x_2, \dots, x_n) } \nonumber \\ && =
\sum_{i=2}^{n} (-1)^{i} \phi \left\{ (x_1, x_2, \dots, x_{i-1}, x_{i+1},\dots, x_n) \right.
\nonumber \\
&&
- \left. (x_1 \ast x_i, x_2 \ast  x_i, \dots, x_{i-1}\ast  x_i,f( x_{i+1}), \dots, f(x_n)) \right\}
\end{eqnarray*}
for $n \geq 2$
and $\delta_n=0$ for
$n \leq 1$.  Then one gets a notion of chain complex and then homology. \end{example}

\begin{example} {\rm 	Let $\eta$ be the multiplication by $T$ and $\tau$ be the multiplication by $S$ as in item 4 of Example \ref{Examples}.  Then the map $\delta^n$ becomes 
		
		\begin{eqnarray*}
			\lefteqn{
				\delta^n \phi(x_1, x_2, \dots, x_n) } \nonumber \\ && =
			\sum_{i=2}^{n} (-1)^{i}\left \{ T \phi(x_1, x_2, \dots, x_{i-1}, x_{i+1},\dots, x_n) \right.
			\nonumber \\
			&&
			- \left. \phi(x_1 \ast x_i, x_2 \ast  x_i, \dots, x_{i-1}\ast  x_i,f( x_{i+1}), \dots, f(x_n)) \right \} + (-1)^n S\phi(x_2,\dots, x_n).
		\end{eqnarray*}
			In particular, the $1$-cocycle condition is written for a function $\phi: X \rightarrow A$
			as
			$$  T \phi(x)  -S \phi(y) + \phi(x* y)=0.$$
			Note that this means that $\phi: X \rightarrow A$ is a quandle homomorphism.

		For $\psi: X \times X \rightarrow A$, 	the $2$-cocycle condition can be written as
			\begin{eqnarray*}
				\lefteqn{ T  \psi (x_1, x_2) +   \psi (x_1 * x_2, f(x_3)) } \\
				& = & T  \psi (x_1, x_3) + S \psi (x_2, x_3)
				+  \psi (x_1 * x_3, x_2 * x_3).
			\end{eqnarray*}
		} 
\end{example}

Now we give a couple explicit examples.

	\begin{example}
	In this example, we compute  the first and second cohomology groups of the $f$-quandle  $X =\mathbb{Z}_3$ with coefficients in the abelian group $\mathbb{Z}_3$.  The quandle $X=\mathbb{Z}_3$ here is considered with $T= S = 1$ and  $f(x) = 2x$.   A direct computations gives that  $\delta ^1$ is the zero map. Thus $H^1(X=\mathbb{Z}_3,A=\mathbb{Z}_3) $ is $3$-dimensional and a basis is $\{\chi_0, \chi_1, \chi_2\}$.
	  Now to compute the second cohomology group we consider a $2$--cocycle  
	  \[ \phi = \sum_{i, j \in X=\mathbb{Z}_3} \lambda_{(i,j)}\; \chi_{(i,j)},\]
	   where $\chi_{(i,j)}$ denote the characteristic functions and $\lambda_{(i,j)}$ are scalars. Then $\phi$ satisfies the following equation, for all the triples $(i, j, k) \in \mathbb{Z}_3,$
\[ \phi(i,j) + \phi(i+j,2k) -  \phi(i, k) - \phi(j, k) - \phi(i+k, j+k) = 0.\]
  Since $\delta^1 = 0$, then $H^2 = ker \delta^2$.  A direct computation gives that $H^2(X=\mathbb{Z}_3,A=\mathbb{Z}_3)$ is $2-$dimension with a basis $\{ \psi_1, \psi_2\}$, where  \[ \psi_1=- \chi_{(0,1)} + \chi_{(0,2)} - \chi_{(1,0)} + \chi_{(2,0)}, \]
  and \[\psi_2= \chi_{(0,1)} - \chi_{(0,2)} - \chi_{(1,2)} + \chi_{(2,1)}.\]
\end{example}

\begin{example}
	We compute  $H^1$ and $H^2$ with coefficients in the abelian group $\mathbb{Z}_3$ of the $f$-quandle  $X =\mathbb{Z}_3$,  $T=1,  S = 2$ and  $f(x) = 0$.   A direct computations gives  $H^1(\mathbb{Z}_3,\mathbb{Z}_3) $ is $1$-dimensional with  a basis  $\chi_0+ \chi_1+\chi_2$.
	  Now let consider a $2$--cocycle  $ \phi = \sum_{i, j \in \mathbb{Z}_3} \lambda_{(i,j)} \chi_{(i,j)}$ where $\chi (i,j)$ denote the characteristic function. Then $\phi$ satisfies the following equation, for all the triples $(i, j, k) \in \mathbb{Z}_3$,
\[ \phi(i,j) + \phi(i+j,0) -  \phi(i, k) - \phi(j, k) - \phi(i+k, j+k) = 0.\]
  Therefore a direct computation gives that $H^2$ is $3-$dimension with a basis $\{ \psi_1, \psi_2, \psi_3\}$, where $\psi_1 =  \chi_{(0,1)} + \chi_{(0,2)} - \chi_{(2,1)}, \psi_2= \chi_{(0,1)} - \chi_{(0,2)} - \chi_{(1,0)} + \chi_{(2,0)} $ and $\psi_3 =  \chi_{(0,1)} + \chi_{(0,2)} + \chi_{(2,1)}.$
\end{example}

\begin{remark}

In this remark, we make a short comment on connections to Knot Theory.  We will give a diagramatic notion for the twisted operator as follows and then show how it works for Reidemeister moves:\\
The relation between quandles and knots leads to the natural question of whether $f$-quandles could be used to define similar knot invariants. The introduction of the $f$-map in the axioms however causes the standard labeling scheme for the crossing diagram to fail to produce labelings invariant under the Reidemeister moves. Instead we found the following crossing diagram to be of greater interest.

\[\includegraphics[scale=0.25]{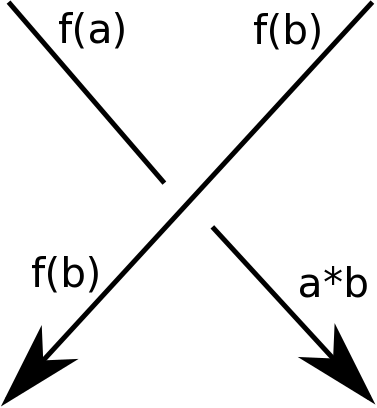} \]

\noindent Using this diagram we were able to show an invariance of the labeling scheme under the Reidemeister moves as show below.
\[\includegraphics[scale=0.25]{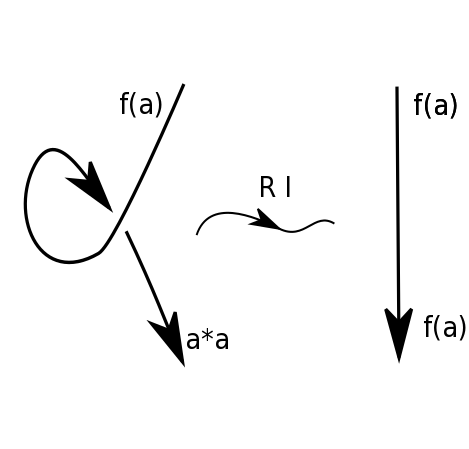} \quad \quad \quad \quad  \quad\includegraphics[scale=0.25]{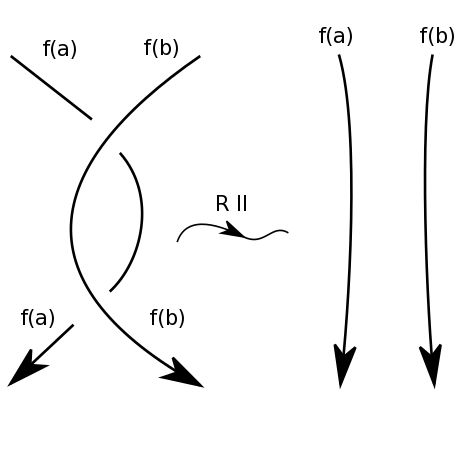}
\]
\[
\includegraphics[scale=0.45]{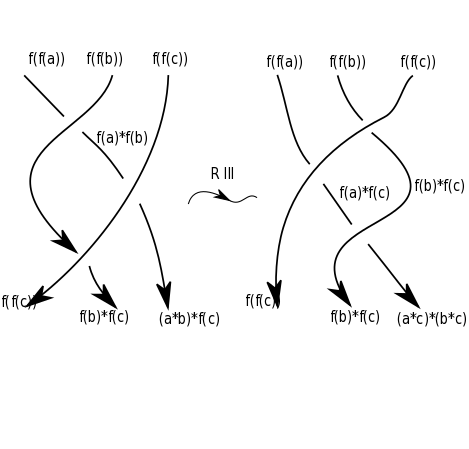}
\]

\noindent Unfortunately, we see that change of label at an under crossing is not a closed operation unless the $f$-map is bijective. That is, for a $f$-quandle $(X,*,f)$, the label of the arc entering a crossing must be an element of $f(X)$, while the label of the out going arc need not be. Indeed, to ensure a consistent labeling, it is clear one must only use labels from $f^n(X)$, where $n$ is such that $f^n(X)=f^{n+m}(X), \forall m$, to ensure one may return to the same labeling upon completing a circuit of the knot. If such a subset is selected though, one is restricted to a sub-$f$-quandle on which the $f$-map is bijective.  

\noindent While these issues are resolved when labeling via a $f$-quandle with a bijective $f$-map, Remark 3.2 shows that this reduces the labeling by the $f$-quandle to precisely that of a standard quandle.
\end{remark}

\end{document}